\documentclass[9.0pt, a4paper]{article}
\usepackage{amsmath}
\usepackage{fancyhdr}

\usepackage{color}
\usepackage{amsmath}
\usepackage{amsfonts}
\usepackage{dsfont}
\usepackage{mathrsfs}
\usepackage{latexsym}
\usepackage{graphicx}
\usepackage{epstopdf}
\usepackage{amssymb}
\usepackage{indentfirst}
\usepackage{subfigure}
\usepackage{booktabs}
\usepackage{authblk}
\usepackage{cite}
\usepackage{float}
\usepackage[numbers,sort&compress]{natbib}
\usepackage{caption}
\usepackage{booktabs}
\usepackage{threeparttable}
\usepackage{multicol}
\usepackage{multirow}
\usepackage{geometry}
\usepackage{lineno}
\usepackage{amsbsy}
\usepackage{bm}
\usepackage{dcolumn}
\usepackage{mathrsfs}
\usepackage[title]{appendix}
\usepackage{stmaryrd}
\usepackage{tikz}
\tikzset{>=stealth}

\newtheorem{theorem}{Theorem}[section]
\newtheorem{lemma}[theorem]{Lemma}

\newtheorem{example}[theorem]{Example}

\numberwithin{equation}{section}

\def\qed{\hfill$\Box$\vspace{8pt}}

\begin{document}
\captionsetup[figure]{labelfont={bf},labelformat={default},labelsep=period,name={Fig.}}

\title{Pointwise A Posteriori Error Estimators for Elliptic Eigenvalue Problems}\date{}
\author{Zhenglei Li}
\author{Qigang Liang}
\author[1,*]{Xuejun Xu}
\affil[1]{\small School of Mathematical Science, Tongji University, Shanghai 200092, China and Key Laboratory of Intelligent Computing and Applications (Tongji University), Ministry of Education (2211186@tongji.edu.cn, qigang$\_$liang@tongji.edu.cn, xuxj@tongji.edu.cn)}

\maketitle

{\bf{Abstract:}}\ \ In this work, we propose and analyze a pointwise a posteriori error estimator for simple eigenvalues of elliptic eigenvalue problems with adaptive finite element methods (AFEMs). We prove the reliability and efficiency  of the residual-type a posteriori error estimator in the sense of $L^{\infty}$-norm, up to a logarithmic factor of the mesh size. For theoretical analysis, we also propose a theoretical and non-computable  estimator, and then analyze the relationship between computable estimator and theoretical estimator. A key ingredient in the a posteriori error analysis is some new estimates for regularized derivative Green's functions. This methodology is also extended to the nonconforming finite element approximations. Some numerical experiments verify our theoretical results.


\section{Introduction} 

This paper considers an a posteriori error estimator for simple eigenvalues of second order elliptic eigenvalue problems with adaptive finite element methods (AFEMs) in the sense of $L^{\infty}$-norm. For approximating second order elliptic partial differential equations, classical finite element methods may perform badly because the exact solution often has local singularities, which leads to the loss of accuracy. In order to cure this problem, AFEMs are proposed in \cite{MR483395}. Over few decades, AFEMs have been extensively studied and successfully brought into  practice (see \cite{MR745088,AINSWORTH19971,MR2194587,MR2231704,i28} and references therein). For elliptic eigenvalue problems,  the 
authors in \cite{i30,i34,i32,i35} analyze some a posteriori error estimates and obtain the reliability and efficiency of these a posteriori error estimators. Moreover, the convergence and optimal complexity of AFEMs for elliptic eigenvalue problems are derived in \cite{MR2430983,MR3347459}.

\par  The a posteriori errors above are measured in the sense of energy norm or $L^2$-norm, thus the discrete solution converges to the exact solution in the sense of average integration. However, it is worth noting that the analysis of pointwise (i.e., $L^{\infty}$-norm) error may be important in many applications (see \cite{MR1839794,MR3549870,MR1079028} and references therein). Therefore, it is significant to propose and analyze the a posteriori error estimator in the sense of $L^{\infty}$-norm, which often poses a great challenge.

\par In recent decades, the studies of the a priori error in the sense of $L^{\infty}$-norm by finite element methods for second order elliptic source problems have been widely developed (see \cite{i2,i3,i4,i7,construction2,i13,MR4435940} and references therein).
For the a posteriori error analysis, the author in \cite{1} derives pointwise a posteriori error estimators over general polygonal domains in $\mathbb{R}^2$ and proves their equivalence to the a priori error (the $L^{\infty}$-norm of $u-u_{h}$). Then these results are extended to three-dimensional and nonconforming cases in \cite{2}. The key ingredient in \cite{1} to prove the equivalence is to give an a priori $W^{2,1}$-estimate for regularized Green's functions. As for the pointwsie a priori error analysis for eigenvalue problems, the sharp convergence rate is obtained in convex domains under the assumption that $u\in W^{2,\infty}(\Omega)$ in \cite{i36}. However, to the best of our knowledge, the study of the pointwise a posteriori error for eigenvalue problems is not available in the literature. 

\par In this paper, we present a pointwise a posteriori error estimator for eigenfunction corresponding to simple eigenvalues, and prove its reliability and efficiency in the sense of   $L^{\infty}$-norm for second order elliptic eigenvalue problems. We must emphasize that it is non-trivial to analyze the reliability  of the pointwise a posteriori error estimator. First, a crucial difference between eigenvalue problems and source problems is that there is no Galerkin orthogonality in the finite element approximation for eigenvalue problems, which increases an extra term in the residual equation. Second, since the a posteriori error estimator is analyzed in the sense of $L^{\infty}$-norm, many estimate techniques developed in a separable Hilbert space in the sense of energy inner product (or $L^{2}$ inner product) are not applicable.

\par To overcome these difficulties and obtain the reliability and efficiency of the a posteriori error estimator for simple eigenvalues in the sense of $L^{\infty}$-norm, we derive some new estimates. A key ingredient in this paper is to obtain a proper estimate for the Galerkin approximation of  regularized derivative Green's functions in the sense of $L^1$-norm. To get it, we use the Aubin-Nitsche argument and take a sign function as the source term on adaptive mesh. By exploiting the $L^{p}$-estimate argument for the gradient of solutions of elliptic partial differential equations and combining some new estimate results, we may prove the  reliability of the pointwise a posteriori error estimator. Then a modification of a local argument (see \cite{1,i20}) gives the efficiency of the a posteriori error estimator. Specifically, we obtain the following result:
\begin{equation}\label{Inequ_pos_err}
   \eta \lesssim \Vert e_h\Vert_{L^{\infty}(\Omega)} \lesssim \vert \ln{h}\vert \eta,
\end{equation}
where $a\lesssim(\gtrsim)\ b$ means $a\leq(\geq)\ Cb$, the letter $C$ is generic positive constant independent of
the mesh size, possibly different at each occurrence, $e_h:=u-u_{h}$ with $u$ and $u_{h}$ denoting the eigenfunction and its adaptive finite element approximation, and $\eta$ is the pointwise a posteriori error estimator (defined in section 4).  In addition, we extend this result to the Crouzeix-Raviart (CR) nonconforming finite element.

\par This paper is organized as follows. In section 2, we introduce the model problem and some mathematical notations. Then we review some classical and useful a priori estimates and derive a priori estimates for regularized Green's functions, which improves the result in \cite{1}. In section 3, we prove the reliability and efficiency of the pointwise a posteriori error estimator. 
In section 4, we extend the results to the CR nonconforming finite element.  Some numerical experiments are given to verify our findings in section 5 and the conclusion is presented in section 6.

\section{Preliminaries}

\subsection{Notation}
Let $\Omega\subset \mathbb{R}^2$ be a polygonal domain without restrictions on the size of the interior angles.  Throughout this paper, for any subset $\widetilde{\Omega} \subset \Omega$, we use the standard notations for the Sobolev spaces $W^{m,p}(\widetilde{\Omega})$ ($H^m(\widetilde{\Omega})$) and $W_{0}^{m,p}(\widetilde{\Omega})$($H_0^m(\widetilde{\Omega})$). Moreover, we denote by $L^p(\widetilde{\Omega})$ and $L^p(\Gamma)$ the usual Lebesgue spaces equipped with the standard norms $\Vert \cdot\Vert_{L^p(\widetilde{\Omega})}$ and  $\Vert \cdot\Vert_{L^p(\Gamma)}$, where  $\Gamma \subset \overline{\Omega}$ is a Lipschitz curve. For notational convenience, we denote by 
\begin{equation*}
	(\phi ,\psi)_{\widetilde{\Omega}}:=\int_{\widetilde{\Omega}} \phi \psi\ dx,\ \ \ \ \langle \phi ,\psi \rangle_\Gamma:=\int_\Gamma\phi \psi\ ds.
\end{equation*}
If $\widetilde{\Omega}=\Omega$, we drop the subscript $\widetilde{\Omega}$ in $(\phi ,\psi)_{\widetilde{\Omega}}$.

\par We consider the model problem
\begin{align}\label{Laplace}
	\begin{cases}
		-\Delta u=\lambda u \ \ &\text{in} \ \Omega,\\
		u=0 \ \ &\text{on} \  \partial \Omega.
	\end{cases}
\end{align}
The variational form of \eqref{Laplace} is to find $(\lambda,u)\in \mathbb{R} \times V (:=H_{0}^{1}(\Omega)) $ such that
\begin{equation}\label{Laplace_variation}
  a(u,v)=\lambda(u,v)\ \ \ \ \forall\ v \in V,\\
\end{equation}
where 
\begin{equation*}
	a(v ,w)=\int_\Omega \nabla v \cdot \nabla w \ dx \ \ \ \ \forall\ v ,w \in V.
\end{equation*}
Define an linear operator $T$ such that for any $ f\in L^2(\Omega)$,
\begin{align}\label{defT}
    a(Tf,v)=(f,v)\ \ \ \ \forall\ v\in V.
\end{align}
Then we may obtain the equivalent form of \eqref{Laplace_variation} as
 $Tu=\lambda^{-1}u$.
\par Let $V_h$ be the space of continuous piecewise linear polynomials vanishing on $\partial \Omega$ associated with a shape-regular partition $\mathcal{T}_h$. The discrete eigenpair $(\lambda_h,u_h)$ satisfies
\begin{equation}\label{def_u_h}
a(u_h,v_h)=\lambda_h(u_h,v_h) \ \ \ \ \forall \ v_h\in V_h.
\end{equation}
We set $ h=\max_{K \in \mathcal{T}_h} h_K,$ where $h_K$ is the diameter of element $K \in \mathcal{T}_h$, and assume that there exists a positive constant $\gamma$ such that
\begin{equation}\label{size}
    h\leq h_K^{\gamma} \ \ \ \  \forall\ K\in \mathcal{T}_h.
\end{equation}
Let $\mathcal{E}_h$ be the set of all edges of elements contained in $\mathcal{T}_h$. For any element $K$, we use $\omega_K$ to denote the union of all elements which share common edge with $K$.  We now introduce the jump operator. Given an edge $E\subset\Omega$ shared by two elements $K_{+}$ and $K_{-}$, i.e., $E=\partial K_+ \cap \partial K_-$,
 for any piecewise function $v$, we define  
 \begin{equation*}
     [v]_E=(v|_{K_+})|_E-(v|_{K_-})|_E
 \end{equation*}
 and
 \begin{equation*}
     [\nabla v]_E=(\nabla v|_{K_+})|_E-(\nabla v|_{K_-})|_E.
 \end{equation*}
 By convention, if $E\in \mathcal{E}_h$ is a boundary edge, we set
\begin{equation*}
    [v]_E=v|_{E}, \ \ \ \ [\nabla v]_E=\nabla v |_{E}.
\end{equation*}
For simplicity, we denote by 
 \begin{align*}
    [\frac{\partial v}{\partial {n} 
 }]_E &=[\nabla v]_E \cdot {n}_E,
 \end{align*}
 where ${n}_E$ is the unit outer normal vector of $E$. 
The choice of the particular normal is arbitrary, but is considered to be fixed once and for all. In subsequent sections, we will drop the subscript $E$ without causing confusion. 
 
\par Next we introduce the definitions of some operators. For the space $V_h$, we define the Galerkin projection $G_h$: $H_0^1(\Omega)\rightarrow V_h$ such that
 \begin{equation}\label{def_G_h}
     a(G_hv,v_h)=a(v,v_h)\ \ \ \ \forall \ v_h\in V_h.
 \end{equation}
The usual Lagrange interpolation  $I_h$:
 $C^0(\overline{\Omega})\rightarrow V_h$ has the local approximation properties \cite{susanne} :
\begin{equation}\label{app_I_h}
    \Vert v-I_hv\Vert_{W^{m,p}(K)}\lesssim h_K^{2-m}\vert v \vert_{W^{2,p}(K)}, \ \ \text{for}\  m=0,1, \ p\geq 1,\  v\in W^{2,p}(K).
\end{equation}
Then we define a discrete operator $T_h$ such that for any $f\in L^2(\Omega)$,
\begin{equation}\label{defT_h}
a(T_h f,v_h)=(f,v_h)\ \ \ \ \forall \ v_h\in V_h.
\end{equation}

\subsection{Some a priori estimates}
In this subsection, we first introduce a regularized Green's function and then review some useful a priori estimates. 
\par Given a fixed point $x_0\in \overline{\Omega}$, we may define a smooth function $\delta_0$ satisfying the following properties \cite{1}:
 %Next we will use the method in \cite{1} to deal with it. 
\begin{equation}\label{prodelta1}
    \text{supp}\  \delta_0 \subset B:=\{x\in\Omega : \vert x-x_0\vert < \rho \},\ \ \int_{\Omega} \delta_0 \ dx=1,
\end{equation}

\begin{equation}\label{prodelta2}
     0\leq \delta_0 \lesssim \rho^{-2},\ \ \Vert  \delta_0 \Vert_{W^{k,\infty}(\Omega)}\lesssim \rho^{-k-2}, \ k=1,2,....
\end{equation}
where $\rho =h^\beta$ with an underdetermined parameter $\beta\geq 1$. Then we may define regularized Green's function $g_0$ satisfying
\begin{align}\label{defg}
    \begin{cases}
		-\Delta g_0=\delta_0\ \ \ \ &\text{in} \ \Omega,\\
		g_0=0\ \ \ \ &\text{on} \ \partial\Omega.       
    \end{cases}\ \  
\end{align}

\par It is known that the following estimate holds for some $p>1$ under appropriate assumptions on the boundary condition \cite{6,5}: 
\begin{equation}\label{priori}
    \vert v \vert_{W^{2,p}(\Omega)}\leq C_p \Vert \Delta v \Vert_{L^p(\Omega)}  \ \ \ \ \forall \ v \in W^{2,p}(\Omega)\cap W_0^{1,p}(\Omega).
\end{equation}
For a smooth domain, it is well known that $C_p=\frac{C}{p-1}$ when $p\rightarrow 1$ (see \cite{4}). The following estimates are straightforward in view of \eqref{prodelta2}
\begin{equation*}\label{smooth1}
    \vert g_0 \vert_{W^{2,p}(\Omega)}\lesssim \frac{\rho^{2(\frac{1}{p}-1)}}{p-1}.
\end{equation*}
Choosing $p=1+\vert \ln\rho\vert ^{-1}$, we  may obtain 
\begin{equation}\label{smooth2}
     \vert g_0 \vert_{W^{2,1}(\Omega)}\lesssim \vert \ln\rho\vert.
\end{equation}
For a polygon without restrictions on the size of the interior angles, \eqref{priori} holds for $1<p<\frac{4}{3}$ \cite{6,5}, but the dependence on $p$ of the constant $C_p$ in \eqref{priori} is not known explicitly. In \cite{1}, a slightly weaker estimate is derived as
\begin{equation}\label{w_poly_g0}
    \vert g_0 \vert_{W^{2,1}(\Omega)}\lesssim \vert \ln \rho \vert^2.
\end{equation}
However, we may prove that \eqref{smooth2} still holds for the polygonal domain (see Appendix of this paper).

\section{A posteriori error analysis}
In this section, we first give some a priori estimates for a regularized derivative Green's function (defined in \eqref{def_g1}). Then in subsection 3.2, for convenience of theoretical analysis, we introduce two pointwise estimators, including computable and theoretical estimators. Specifically, given an eigenpair $(\lambda,u)$ and  the corresponding discrete eigenpair $(\lambda_h,u_h)$, the computable estimator is defined by
\begin{equation*}
\eta:=\max_{K\in\mathcal{T}_h}\left(h_K^2\Vert \lambda_hu_h\Vert_{L^{\infty}(K)}+h_K\Vert[\frac{\partial u_h}{\partial \bm{n}}]\Vert_{L^{\infty}(\partial K\backslash\partial \Omega)}\right),
\end{equation*}
and  the theoretical estimator is defined by
\begin{equation*}
\eta^*:=\max_{K\in\mathcal{T}_h}\left(h_K^2\Vert u_h\Vert_{L^{\infty}(K)}+h_K\Vert[\frac{\partial T_{h}u}{\partial \bm{n}}]\Vert_{L^{\infty}(\partial K\backslash\partial \Omega)}\right),   
\end{equation*}
where $T_h$ is defined in \eqref{defT_h}. We provide an upper bound of $\Vert e_h\Vert_{L^{\infty}(\Omega)}$ controlled by the sum of these two estimators in this subsection,  where $e_h:=u-u_h$ with $u$ being the exact eigenfunction $(\|u\|_{L^{2}(\Omega)}=1)$ and $u_h$ being the discrete eigenfunction $(\|u_{h}\|_{L^{2}(\Omega)}=1)$ corresponding to $u$. In subsection 3.3, we analyze the relationship between $\eta$ and $\eta^*$. Finally, we shall prove the main result, namely the reliability and efficiency of $\eta$.
\begin{theorem}\label{main result}
     Let $\Omega$ be an arbitrary polygonal domain, $(\lambda,u)$ is an eigenpair of   \eqref{Laplace}, $(\lambda_h,u_h)$ is the corresponding discrete eigenpair, then the pointwise a posteriori error estimator $\eta$ has the  reliability and efficiency, up to a logarithmic factor
of a small enough mesh size $h$, namely
\begin{equation}\label{Equivalence111111}
    \eta \lesssim ||e_h||_{L^{\infty}(\Omega)}\lesssim |\ln h| \eta.
\end{equation}
\end{theorem}

\subsection{Some estimates for regularized derivative Green's functions}
\par For any element $K\in \mathcal{T}_h$, we use $x_K$ to denote the barycenter of $K$. Then there exists a function $\delta_{1}\in C_0^{\infty}(K)$ (the construction for it can be found in \cite{construction2,construction1}) satisfying
\begin{equation}\label{prodT}
	 (\delta_1, v_h)_K=v_h(x_K) \ \ \forall\  v_h\in P_1(K),\  \Vert \delta_{1}\Vert_{W^{k,\infty}(\Omega)}\lesssim h_K^{-k-2},k=0,1,2,....
\end{equation} 
We define a regularized derivative Green's function $g_1$ as follows:
\begin{align}\label{def_g1}
	\begin{cases}
		-\Delta g_{1}=\nabla\delta_{1}\cdot \boldsymbol{\nu},\ \  &\text{in} \ \Omega,\\
		g_{1}=0,\ \ &\text{on} \ \partial \Omega,
	\end{cases}
\end{align}
where $\boldsymbol{\nu}$ is an arbitrary direction vector. The introduction of the regularized derivative Green's functions plays a crucial role in analyzing the relationship between $\eta$ and $\eta^*$. In subsequent analysis, we need some estimates for $g_1$ and $G_hg_1$ (recall that $G_h$ stands for the Galerkin projection) in different norms. We introduce a result about the gradient estimate in \cite{MR2141694}, which is useful for estimating the $W^{1,p}$-norm of $g_1$.
\begin{lemma}\label{lemma gradient estimate}
    Let $D$ be a Lipschitz domain. Let $U$ solve  the Dirichlet problem $-\Delta U=\text{div}F$ in $D$ with the homogeneous  Dirichlet condition on  $\partial D$, where $\boldsymbol{F}\in (L^p(D))^2$. Then the following $W^{1,p}$-estimate holds
\begin{equation}\label{gradient}
    ||{U}||_{W^{1,p}(D)}\lesssim ||{\boldsymbol{F}}||_{L^{p}(D)},
\end{equation}
provided that $\frac{4}{3}- \varepsilon<p< 4+\varepsilon$, where $\varepsilon>0$ depends only on $D$.
\end{lemma}

\par Now, we give a $W^{1,p}$-estimate for $g_1$.
\begin{lemma}\label{estg_1}
      Let $\delta_1$ and $g_1$ be defined as in \eqref{prodT} and \eqref{def_g1}, then the following estimate holds    \begin{equation}\label{estimate_g1_1p}
        \vert g_1\vert_{W^{1,p}(\Omega)} \lesssim h_K^{-\frac{2}{q}},
    \end{equation}
    where $\frac{4}{3}-\varepsilon_1<p\leq 2$ and $\frac{1}{p}+\frac{1}{q}=1$ and $\varepsilon_1>0$ depends only on $\Omega$.
\end{lemma}
\begin{proof}
Recall that $\delta_1\in C_0^{\infty}(K)$. We may construct a function $\boldsymbol{F}$ by the indefinite integral, such that 
\begin{equation}\label{pro F}
  \boldsymbol{F}\in W_0^{1,\infty}(K),\ \ \ \ \text{div} \boldsymbol{F}=\nabla\delta_{1}\cdot \boldsymbol{\nu},\ \ \ \ \ 
  ||{\boldsymbol{F}}||_{L^{\infty}(\Omega)}\lesssim h_K^{-2}.  
\end{equation}
Using Lemma \ref{lemma gradient estimate} and \eqref{pro F}, we have
\begin{align*}
	\vert g_{1}\vert_{W^{1,p}(\Omega)}&\lesssim ||{\boldsymbol{F}}||_{L^{p}(\Omega)}
 \lesssim h_K^{-\frac{2}{q}},
\end{align*}
which completes the proof. \qed
\end{proof}
 
With the help of Lemma \ref{estg_1}, we may obtain some estimates for $G_hg_1$.
\begin{lemma}\label{estG_hg_1}
    It holds that
\begin{equation}\label{estimate_g1_H1}
        \vert G_hg_1\vert_{H^1(\Omega)}\lesssim h_K^{-1}
    \end{equation}
    and
    \begin{equation}\label{estimate_g1_L1}
        \Vert G_hg_1\Vert_{L^1(\Omega)}\lesssim h^{\frac{1}{3}}h_K^{-1}.
    \end{equation}
\end{lemma}
\begin{proof}
     In view of \eqref{prodT} and \eqref{def_g1}, we have
\begin{align*}
    \vert G_hg_1\vert_{H^1(\Omega)}^2 &=a(G_hg_1,G_hg_1)=(G_hg_1,\nabla \delta_1 \cdot \boldsymbol{\nu})\\
    &=(\nabla G_hg_1\cdot \boldsymbol{\nu}, \delta_1)\lesssim
     h_K^{-1}\vert G_hg_1\vert_{H^1(\Omega)},
\end{align*}
which yields \eqref{estimate_g1_H1}.  The proof of \eqref{estimate_g1_L1} is more complicated. Let 
 \begin{align*}
 	\varphi =sgn(G_hg_1)=\begin{cases}
 	 1&G_h g_1> 0,\\
   0 &G_h g_1= 0,\\
 	 -1&G_h g_1<0.
 	\end{cases}
 \end{align*} 
We consider an auxiliary problem
 \begin{align*}
 	\begin{cases}
 		-\Delta \psi=\varphi\ \  &\text{in} \  \Omega,\\
 		\psi=0\ \ &\text{on} \ \partial \Omega.
 	\end{cases}
 \end{align*}
Let $r_0=\frac{\pi}{\theta}$, where ${\theta}$ is the largest angle among the interior corners of $\Omega$. It is well known that  $\psi\in H^{1+r}(\Omega)\cap H_0^1(\Omega)$ for all $r\in (0,r_0)$, and
\begin{equation}
	\Vert \psi \Vert_{H^{1+r}(\Omega)}\lesssim \Vert \Delta \psi\Vert_{L^{2}(\Omega)}.
\end{equation}
Due to the interpolation theory and the H{\"o}lder inequality, we have
 \begin{align}\label{eG_h}
\begin{aligned}
  \Vert G_hg_1\Vert_{L^1(\Omega)}
  &=(G_hg_1,\varphi)=a(G_hg_1,\psi)\\
  &=a(G_hg_1,\psi-I_h \psi)+a(G_hg_1,I_h\psi)\\
  &=a(G_hg_1,\psi-I_h \psi)+a(g_1,I_h\psi)\\
 	&\lesssim h^r\vert G_hg_1\vert_{H^1(\Omega)}+\vert g_1\vert_{W^{1,p}(\Omega)}\vert \psi\vert_{W^{1,p^*}(\Omega)},
\end{aligned}
 \end{align}
 where $\frac{4}{3}<p< 2$ and $\frac{1}{p}+\frac{1}{p^*}=1$. 
By the Sobolev embedding theorem and the regularity estimate, we obtain
\begin{equation}
	\vert \psi \vert_{W^{1,p^*}(\Omega)}\lesssim \Vert \psi \Vert_{W^{2,q}(\Omega)}\lesssim \Vert \varphi \Vert_{L^{q}(\Omega)}\lesssim 1,
\end{equation}
where $1<q=\frac{2p^*}{p^*+2}<\frac{4}{3}$. Then combining \eqref{estimate_g1_1p} and \eqref{eG_h}, we complete the proof on choosing $s=\frac{1}{3}$ (note that $r_0>\frac{1}{2} 
$ and $2<p^*<4$).\qed
\end{proof}

\subsection{Upper bound}
 In this subsection, we provide an upper bound of $||e_h||_{L^{\infty}(\Omega)}$
controlled by the sum of two estimators ($\eta$ and $\eta^*$).  For simplicity, we assume that each element $K\in \mathcal{T}_h$ is a closed set.
Let $x_0\in \overline{\Omega}$ satisfy 
\begin{equation}\label{defK0}
    \vert e_h(x_0)\vert=\Vert e_h\Vert_{L^{\infty}(\Omega)}.
\end{equation}
% In this paper, we assume that $\Vert e_h\Vert_{L^{\infty}(\Omega)}>0$  (otherwise $u\equiv 0$).

% \subsection{Association between the a priori error and the variational form}
\par For the subsequent analysis, our first step is to prove 
\begin{equation}\label{assosiation111}
    \Vert e_h\Vert_{L^{\infty}(\Omega)}\lesssim \vert (\delta_0,e_h)\vert,
\end{equation}
where $\delta_0$ is defined in \eqref{prodelta1} and \eqref{prodelta2}. 
We may use a similar argument as in \cite{1} to get the following lemma. For simplicity, we omit the detailed proof. 
\begin{lemma}\label{lem nochetto}
Let $0< \alpha\leq 1$ be the H{\"o}lder continuity exponent of the eigenfunction of \eqref{Laplace}. Then there exists a positive number $h_0$ such that for any $h\leq h_0$, the following estimate holds
\begin{equation}\label{associate}
    \Vert e_h\Vert_{L^{\infty}(\Omega)}\lesssim
    \vert (\delta_0,e_h)\vert + h_{K_0}^{\alpha\beta},
\end{equation}
where $K_0$ is an element containing $x_0$ and $\beta\geq 1$.
\end{lemma}

\par In order to drop the second term in \eqref{associate}, we give an a priori lower bound of $\Vert e_h\Vert_{L^{\infty}(\Omega)}$, whose proof may be found in Appendix.  

\begin{lemma}\label{Baohe}
There exists a positive number $h_0$ such that for any $h\leq h_0$ there exists an element $K^*\in \mathcal{T}_h$ satisfying
    \begin{equation*}
         \Vert e_h \Vert_{L^{\infty}(K^*)}\gtrsim h_{K^*}^2.
    \end{equation*}
\end{lemma}
 Thanks to Lemmas \ref{lem nochetto} and \ref{Baohe}, together with \eqref{size}, choosing $\beta > \frac{2\gamma}{\alpha}$, we may deduce \eqref{assosiation111}. 
\par Next we shall derive the upper bound of $\Vert e_h\Vert_{L^{\infty}(\Omega)}$ from the residual error equation. By integrating by parts, $e_h$ satisfies the error equation
\begin{align*}
    a(e_h,v)&=\sum_{K\in \mathcal{T}_h}  \left( (\lambda u ,v-v_h)_K+\langle  \frac{\partial u_h}{\partial \bm{n}},v-v_h \rangle_{\partial K}) \right) + (\lambda u-\lambda_h u_h,v_h)\\
    &= \sum_{K\in \mathcal{T}_h}\left((\lambda_h u_h ,v-v_h)_K+\langle  \frac{\partial u_h}{\partial \bm{n}},v-v_h \rangle_{\partial K})\right) + (\lambda u-\lambda_h u_h,v),
\end{align*}
where $v\in V$, $v_h\in V_h$. Taking $v=g_0$ and $v_h=I_h g_0$, in view of  \eqref{smooth2} and  \eqref{assosiation111}, with the interpolation theory and the trace theorem on $K$, we deduce
\begin{align}\label{errorform1} 
\begin{aligned}
\Vert e_h\Vert_{L^{\infty}(\Omega)}
&\lesssim \vert (\delta_0, e_h)\vert=\vert a(e_h,g_0)\vert\\
   &\lesssim \vert \ln{h}\vert  \max_{K\in\mathcal{T}_h}\left(h_K^2\Vert \lambda_h u_h\Vert_{L^{\infty}(K)}+h_K\Vert[\frac{\partial u_h}{\partial \bm{n}}]\Vert_{L^{\infty}(\partial K\backslash \partial\Omega)}\right)\\
    &\ \ \ \ +\vert ( \lambda u-\lambda_h u_{h},g_0 )\vert.
\end{aligned}
\end{align}
Noting that for any $v\in W^{2,1}(\Omega)\cap H_{0}^{1}(\Omega)$, it holds that
\begin{align*}
|v|_{H^{1}(\Omega)}^{2}
&=(\nabla{v},\nabla{v})
=(-\Delta{v},v)\leq ||\Delta{v}||_{L^{1}(\Omega)}||v||_{L^{\infty}(\Omega)}\\
&\lesssim |v|_{W^{2,1}(\Omega)} ||v||_{W^{2,1}(\Omega)}
\lesssim |v|_{W^{2,1}(\Omega)}\left(||\nabla{v}||_{L^{1}(\Omega)}+|v|_{W^{2,1}(\Omega)}\right)\\
&\lesssim  |v|_{W^{2,1}(\Omega)}\left(||\nabla{v}||_{L^{2}(\Omega)}+|v|_{W^{2,1}(\Omega)}\right),
\end{align*}
which, together with the Young inequality, yields
\begin{align}\label{v_estimate_H1}
|v|_{H^{1}(\Omega)}\lesssim |v|_{W^{2,1}(\Omega)}.
\end{align}
Since $g_0\in W^{2,1}(\Omega)\cap H_{0}^{1}(\Omega)$, by \eqref{v_estimate_H1} and the Poincar\'e inequality, we get 
$$\Vert g_0\Vert_{L^2(\Omega)}\lesssim \vert g_0\vert_{W^{2,1}(\Omega)}\lesssim \vert \ln{h}\vert.$$
Inserting this result into \eqref{errorform1}, we obtain
\begin{equation}\label{errorform2}
     \Vert e_h\Vert_{L^{\infty}(\Omega)}\lesssim \vert \ln{h}\vert (\eta+ \Vert \lambda u-\lambda_h u_h\Vert_{L^2(\Omega)}).
\end{equation}
\par In the subsequence, we shall estimate the term $\Vert \lambda u-\lambda_h u_h\Vert_{L^2(\Omega)}$ by using Babu\v{s}ka-Osborn theory in \cite{MR962210}.

\begin{lemma}\label{lem lam-lamh}
There exists a positive number $h_0$ such that for any $h\leq h_0$, we have
    \begin{equation}\label{lamu-lamhuh}
        \Vert \lambda u-\lambda_h u_h\Vert_{L^2(\Omega)}\lesssim \eta^* +h^2\Vert e_h\Vert_{L^{\infty}(\Omega)}.
    \end{equation}
\end{lemma}
\begin{proof}    
 From \cite{MR962210},  we know that
    \begin{equation}\label{lam-lamh}
        \lambda_h-\lambda \lesssim \Vert (T-T_h)u\Vert_{L^{2}(\Omega)}
    \end{equation}
and
    \begin{equation}\label{u-uh__}
       \Vert u-u_h\Vert_{L^{2}(\Omega)}\lesssim \Vert (T-T_h)u\Vert_{L^{2}(\Omega)}.
    \end{equation}
The  triangle inequality directly leads to 
\begin{equation}\label{errorform2_1}
    \Vert \lambda u-\lambda_h u_h\Vert_{L^2(\Omega)}\lesssim \Vert (T-T_h)u\Vert_{L^2(\Omega)}.
\end{equation}
Then we consider an auxiliary problem
\begin{align}\label{aubin}
	\begin{cases}
		-\Delta \omega=(T-T_h)u\ \ &\text{in} \ \Omega,\\
		\omega=0\ \ &\text{on} \ \partial \Omega.
	\end{cases}
\end{align}
In view of the definition of $T$ and $T_h$, using a  similar argument as in \eqref{errorform1}, we deduce
\begin{align}\label{align_T_Th_1}
\begin{aligned}
\Vert (T-T_h)u\Vert_{L^{2}(\Omega)}^2
&=a(Tu-T_hu,\omega)=a(Tu-T_hu,\omega -I_h\omega)\\
&=\sum_{K\in\mathcal{T}_h}\left( ( u,\omega-I_h\omega)_K+\langle\frac{\partial T_hu}{\partial \bm{n}},\omega -I_h\omega\rangle_{\partial K}\right)\\
% &\lesssim \sum_{K\in\mathcal{T}_h}\left(h_K^2 \Vert u\Vert_{L^{\infty}(K)}+h_K\Vert[\frac{\partial T_hu}{\partial \bm{n}}]\Vert_{L^{\infty}(\partial K\backslash\partial \Omega)} \right)\vert \omega \vert_{2,1,K}\\
&\lesssim \vert \omega \vert_{W^{2,1}(\Omega)}(\eta^*+h^2\Vert e_h \Vert_{L^{\infty}(\Omega)}).
\end{aligned}
\end{align}
With the aid of \eqref{priori} and  the H{\"o}lder inequality, for a fixed $p_{0}\in (1,\frac{4}{3})$, we obtain
\begin{equation}\label{est_w}
	\vert \omega\vert_{W^{2,1}(\Omega)}\lesssim \vert \omega\vert_{W^{2,p_0}(\Omega)}\lesssim \Vert (T-T_h)u\Vert_{L^{p_0}(\Omega)}\lesssim \Vert (T-T_h)u\Vert_{L^{2}(\Omega)},
\end{equation}
 which, tother with \eqref{errorform2_1} and \eqref{align_T_Th_1}, yields \eqref{lamu-lamhuh}.\qed
\end{proof}
\par Combining \eqref{errorform2} and Lemma \ref{lem lam-lamh}, we may derive an upper bound for $\Vert e_h\Vert_{L^{\infty}(\Omega)}$.
\begin{theorem}\label{ehsum}
    There exists an $h_0>0$ such that 
    \begin{equation}\label{upper bound}
   \Vert e_h\Vert_{L^{\infty}(\Omega)}\lesssim |\ln h|(\eta+\eta^*)
\end{equation}
   holds whenever $h\leq h_0$.
\end{theorem}
\begin{proof}
Owing to \eqref{errorform2} and  Lemma \ref{lem lam-lamh}, we obtain
\begin{equation*}
    (1-h^2|\ln h|)\Vert e_h\Vert_{L^{\infty}(\Omega)}\lesssim |\ln h|(\eta+\eta^*).
\end{equation*}
Then we derive \eqref{upper bound} noting that $h^2|\ln h|$ is $o(1)$.\qed
\end{proof}

\subsection{Relationship between $\eta$ and $\eta^*$}
\par Now we restrict our attention to the relationship  between $\eta$ and $\eta^*$. In order to give the reliability of $\eta$, since \eqref{upper bound} holds, it suffices to show that $\eta^*$ can be (mainly) bounded by $\eta$. In fact, we have the following estimate for $\eta^*$. 
\begin{lemma}\label{lem eta-eta*__}
It holds that
   \begin{equation}\label{eta-eta*}
    \eta^*\lesssim \eta+  h^{\frac{1}{3}}\Vert e_h\Vert_{L^{\infty}(\Omega)},
\end{equation}
where $s$ is defined in Lemma \ref{estG_hg_1}.
\end{lemma}
\begin{proof}
By the definitions of $\eta$ and $\eta^*$, we know 
\begin{equation}\label{aso_eta-eta^*_2}
    \eta^*\lesssim \eta +\max_{K\in \mathcal{T}_h}h_K\vert T_hu-\lambda_h^{-1}u_h\vert_{W^{1,\infty}(K)}.
\end{equation}
Note that 
\begin{align*}
 |T_hu-\lambda_h^{-1}u_h|_{W^{1,\infty}(K)}\lesssim \sup_{\boldsymbol{\nu}} |(\nabla(T_hu-\lambda_h^{-1}u_h)\cdot \boldsymbol{\nu}
 )(x_K)|,
\end{align*}
 where $\nu$ is an arbitrary direction vector.
Using \eqref{prodT} and \eqref{def_g1},  we deduce 
\begin{align}\label{difference}
\begin{aligned}
	(\nabla(T_hu-\lambda_h^{-1}u_h)\cdot \boldsymbol{\nu}
 )(x_K)&=(T_hu-\lambda_h^{-1}u_h,\nabla \delta_1 \cdot \boldsymbol{\nu})\\
 &=a(T_hu-\lambda_h^{-1}u_h,G_h g_1)\\
	&=(e_h,G_h g_1)\leq \Vert e_h\Vert_{L^{\infty}(\Omega)}\Vert G_h g_1\Vert_{L^1(\Omega)}.
\end{aligned}
\end{align}
Combining Lemma \ref{estG_hg_1} and \eqref{aso_eta-eta^*_2}, we complete the proof.\qed
\end{proof}
\par Now we are in a position to prove the main result, namely the reliability and efficiency of $\eta$.
\par  \noindent{\bf Proof of Theorem \ref{main result}}:  First,  Lemmas \ref{ehsum} and \ref{lem eta-eta*__} lead to the reliability of $\eta$ directly due to the fact that $h^s|\ln h|$ is $o(1)$ for small enough $h$.  For the proof of the efficiency, we modify the local argument in \cite{1,i20}.
\par Let $b_K$ be the usual bubble function, i.e., the polynomial of degree three obtained as the product of the barycentric coordinates of $K$. Taking $v_1=\lambda_h u_h b_K$, with integration by parts, we immediately get
\begin{equation}
	(\lambda u,v_1)=a(e_h,v_1)= (e_h,-\Delta v_1)_K+\langle e_h,\frac{\partial v_1}{\partial \bm{n}}\rangle_{\partial K},
\end{equation}
where we use the fact that $a(u_h,v_1)=0$.
Noting that
$$	 0\leq  b_K\leq 1, \ \ \ \ \Vert b_K\Vert_{W^{2,1}(K)}\lesssim 1,$$
and using the trace theorem, we have
\begin{equation}\label{r1}
	h_K^2 \Vert \lambda_h u_h\Vert_{L^{\infty}(K)}\lesssim \Vert e_h\Vert_{L^{\infty}(K)}+h_K^2\Vert \lambda u-\lambda_h u_h\Vert_{L^{\infty}(K)}.
\end{equation}
\par  Given an interior edge $E\subset \partial K$, we denote by $K_1$ the neighbor element that shares the edge $E$ with $K$. Let $\varphi_E$ be a piecewise quadratic polynomial on $K\cup K_1$, which is one at the midpoint of $E$ and vanishes on the other edges of $K\cup K_1$. Taking $v_2=[\frac{\partial u_h}{\partial \bm{n}}] \varphi_E$, we have
\begin{equation*}
	(\lambda u,v_2)_{K\cup K_1} + \langle v_2,[\frac{\partial u_h}{\partial \bm{n}}]\rangle_E= (e_h,-\Delta v_2)_{K\cup K_1} + \langle e_h,[\frac{\partial v_2}{\partial \bm{n}}]\rangle_E.
\end{equation*} 
Owing to the fact that 
\begin{equation*}
   0\leq \varphi_E\leq 1, \ \ \ \ \Vert \varphi_E\Vert_{W^{2,1}(K\cup K_1)}\lesssim 1,   
\end{equation*}	 
we obtain
\begin{equation}\label{r2}
	h_E\Vert[\frac{\partial u_h}{\partial \bm{n}}]\Vert_{L^{\infty}(E)}\lesssim \Vert e_h\Vert_{L^{\infty}(K\cup K_1)}+ h_K^2\Vert \lambda u\Vert_{L^{\infty}(K\cup K_1)}.
\end{equation}
Adding \eqref{r1} and \eqref{r2} together, 
we may arrive at
\begin{equation*}
	h_K^2\Vert \lambda_hu_h\Vert_{L^{\infty}(K)}+h_K\Vert[\frac{\partial u_h}{\partial \bm{n}}]\Vert_{L^{\infty}(\partial K)} \lesssim \Vert e_h\Vert_{L^{\infty}(\omega_K)}+ h_K^2\Vert \lambda u-\lambda_h u_h\Vert_{L^{\infty}(\omega_K)}.
\end{equation*}
Thus,
\begin{equation}\label{Equation_eta_eh_1}
	\eta \lesssim \Vert e_h\Vert_{L^{\infty}(\Omega)}+
 \max_{K\in \mathcal{T}_h} h_K^2 \Vert \lambda u-\lambda_h u_h\Vert_{L^{\infty}(\omega_K)}.
\end{equation}
Note that
\begin{equation}\label{lamu-lamuh2}
      \Vert \lambda u-\lambda_h u_h\Vert_{L^{\infty}(\Omega)}\lesssim \vert \lambda-\lambda_h\vert +  \Vert e_h\Vert_{L^{\infty}(\Omega)}.
\end{equation}
From the proof argument of Lemmas \ref{lem lam-lamh} and \ref{lem eta-eta*__}, we may infer that 
\begin{equation*}
    |\lambda-\lambda_h|\lesssim \eta +h^{\frac{1}{3}}\Vert e_h\Vert_{L^{\infty}(\Omega)},
\end{equation*}
 % we obtain
 % \begin{equation}
 %     \Vert \lambda u-\lambda_h u_h\Vert_{L^{\infty}(\Omega)}\lesssim \eta +\Vert e_h\Vert_{L^{\infty}(\Omega)},
 % \end{equation}
which, together with \eqref{Equation_eta_eh_1} and \eqref{lamu-lamuh2}, yields the efficiency of $\eta$.  \hfill \qed

\section{Nonconforming elements}
In this section, we may extend the results of the previous section to the nonconforming approximation. The CR finite element is widely used to approximate the model problem \eqref{Laplace}. For simplicity, if no otherwise specified, we will use the same notations as in section 3 for the nonconforming case without causing confusion.  
\par We use $V_h$, consisting of piecewise linear functions continuous at the midpoint of the interior edges and vanishing at the midpoint of
boundary edges, to denote
 the corresponding finite element space associated
with a shape-regular partition  $\mathcal{T}_h$. We denote by $\mathcal{E}_h$ the set of all edges and $\mathcal{E}_h^{I}$  the set of all interior edges.  The discrete eigenpair $(\lambda_h, u_h) $
satisfies 
\begin{equation}
    a_{h}(u_{h},v_{h})=\lambda_h(u_h,v_h) \ \ \ \ \forall \ v_h\in 
 V_h,
\end{equation}
where 
\begin{equation*}
 a_{h}(v,w):=\sum_{K\in \mathcal{T}_h}\int_K \nabla v \cdot \nabla w \ dx\ \ \ \ \forall\ v,w\in V+V_h.   
\end{equation*}
Let $T_h$ be defined similarly as in \eqref{defT_h} for nonconforming case, namely,
\begin{equation*}
a_{h}(T_h f,v_h)=(f,v_h)\ \ \ \ \forall \ v_h\in V_h.
\end{equation*}
We denote by $e_h:=u-u_h$ and let $x_0\in \overline{\Omega}$ satisfy $\vert e_h(x_0)\vert =\Vert e_h\Vert_{L^{\infty}(\Omega)}$. Moreover, we need to slightly modify the definitions of $\eta$ and $\eta^*$, which are defined by
\begin{equation*}
  \eta=\max_{K\in \mathcal{T}_h}\left(h_K^2 \Vert \lambda_h u_h\Vert_{L^{\infty}(K)}+h_K\Vert [\nabla u_h]\Vert_{L^{\infty}(\partial K)}\right)
\end{equation*}
and 
\begin{equation*}\label{eta_*NC}
\eta^*=\max_{K\in\mathcal{T}_h}(h_K^2 \Vert u_h\Vert_{L^{\infty}(K)}+h_K\Vert [ \nabla T_hu]\Vert_{L^{\infty}(\partial K)}).
\end{equation*}
\par First we shall obtain a similar result as \eqref{assosiation111}. To show this, we need to make some modification to the proof argument in section 3 because of the discontinuity of the approximate solution $u_h$ (see the details in \cite{2}). 
\begin{lemma}  Let $\delta_0$ be defined in \eqref{prodelta1} and \eqref{prodelta2}, under the same assumptions of Lemma \ref{lem nochetto}, we have    \begin{equation}\label{asonc}
       \Vert e_h\Vert_{L^{\infty}(\Omega)}\lesssim \vert (e_h,\delta_0)\vert + \max_{E\in \mathcal{E}_h^{I}}\Vert [u_h]\Vert_{L^{\infty}(E)}.
   \end{equation}
\end{lemma}
\par With \eqref{asonc}, we may construct the a posteriori error estimator of the CR nonconforming finite element  discretization  for eigenvalue problems. Recall that $g_0$ is the solution of \eqref{defg} and $I_h$ is the usual $P_{1}$-conforming finite element interpolation. Noting that $I_hg_0 \in V\cap V_h$ and integrating by parts, we deduce
\begin{align*}
        (e_h,-\Delta g_0)&=\sum_{K\in \mathcal{T}_h}\left( (\nabla e_h,\nabla g_0)_K-\langle u_h,\frac{\partial g_0}{\partial \bm{n}}\rangle_{\partial K}\right)\\
        &=\sum_{K\in \mathcal{T}_h}\left( (\nabla e_h,\nabla (g_0-I_hg_0))_K-\langle u_h,\frac{\partial g_0}{\partial \bm{n}}\rangle_{\partial K}\right)+(\lambda u-\lambda_h u_h,I_hg_0)\\   
        &=(\lambda u,g_0-I_hg_0)-\sum_{E\in \mathcal{E}_h}\left(\langle [u_h],\frac{\partial g_0}{\partial \bm{n}}\rangle_E+\langle [\frac{\partial u_h}{\partial \bm{n}}],g_0-I_hg_0\rangle_E\right)\\
        &\ \ \ \ +(\lambda u-\lambda_h u_h,I_hg_0)\\
        &=(\lambda_h u_h,g_0-I_hg_0)-\sum_{E\in \mathcal{E}_h}\left(\langle [u_h],\frac{\partial g_0}{\partial \bm{n}}\rangle_E+\langle [\frac{\partial u_h}{\partial \bm{n}}],g_0-I_hg_0\rangle_E\right)\\
        &\ \ \ \ +(\lambda u-\lambda_h u_h,g_0).
\end{align*}
 With the interpolation theory and the inverse inequality, we immediately obtain
\begin{align}
    \Vert e_h\Vert_{L^{\infty}(\Omega)}\lesssim \vert \ln{h}\vert (\Vert \lambda u-\lambda_h u_h\Vert_{L^2(\Omega)}+\eta).
\end{align}
Since  \eqref{lam-lamh} and \eqref{u-uh__} are still valid for nonconforming cases, we may obtain 
\begin{equation}\label{relation_NC}
    \Vert e_h\Vert_{L^{\infty}(\Omega)}\lesssim \vert \ln{h}\vert  (\Vert (T-T_h)u\Vert_{L^2(\Omega)}+\eta).
\end{equation}

Then we have the following result.
\begin{lemma} 
There exists a positive number $h_0$ such that for any $h\leq h_0$, we have
\begin{equation}\label{T-Th_NC}
      \Vert Tu-T_hu\Vert_{L^2(\Omega)}\lesssim \eta^*+h^2\Vert e_h\Vert_{L^{\infty}(\Omega)}.
\end{equation}    
\end{lemma}

\begin{proof}
We consider the auxiliary problem \eqref{aubin} again with the source term $\omega:=Tu-T_hu$. Then we have 
\begin{align}\label{Inv1}
    \Vert Tu-T_hu\Vert_{L^2(\Omega)}^2=\sum_{K\in \mathcal{T}_h}\langle T_hu,\frac{\partial \omega}{\partial \bm{n}}\rangle_{\partial K}+a_h(Tu-T_hu,\omega).
\end{align}
For the first term in \eqref{Inv1}, using the trace theorem and the continuity of $T_hu$ at the midpoint of the interior edges,  we get
\begin{align}\label{T-ThNC}
\begin{aligned}
    \vert \sum_{K\in \mathcal{T}_h}\langle T_hu,\frac{\partial \omega}{\partial \bm{n}}\rangle_{\partial K}\vert &\lesssim \Vert \omega\Vert_{W^{2,1}(\Omega)}\max_{K\in \mathcal{T}_h}\Vert[T_hu]\Vert_{L^{\infty}(\partial K)}\\
    & \lesssim  \Vert \omega\Vert_{W^{2,1}(\Omega)}\max_{K\in \mathcal{T}_h}h_K\Vert[\nabla T_hu]\Vert_{L^{\infty}(\partial K)}.
    \end{aligned}
\end{align}
For the second term, noting that
\begin{equation}\notag
    a_h(Tu-T_h u,v_h)=0\ \ \ \ \forall \ v_h\in V\cap V_h,
\end{equation}
we have
\begin{align*}
    \vert a_h(Tu-T_h u,\omega)\vert&=\vert a_h(Tu-T_h u,\omega-I_h \omega)\vert \\
    &\lesssim \vert \omega\vert_{W^{2,1}(\Omega)}\max_{K\in \mathcal{T}_h}(h_K^2\Vert u\Vert_{L^{\infty}(K)}+h_K\Vert  [\nabla T_h u]\Vert_{L^{\infty}(\partial K)}).
\end{align*}
Combining \eqref{est_w}, \eqref{T-ThNC} and the above inequality, we complete the proof.\qed
\end{proof}
\par To estimate $\eta^*$, we use a similar argument as in the proof of Lemma \ref{lem eta-eta*__}. In fact, we only need to slightly modify the estimate for $\Vert G_hg_1\Vert_{L^1(\Omega)}$ because of the discontinuity, where $G_h$ denotes the Galerkin approximation with the CR finite element. Let $g_1$ be defined in \eqref{def_g1} and we denote by $\varphi$ the sign of $G_h g_1$. Moreover, let $\psi$ be defined as in the proof of Lemma \ref{estG_hg_1}.  Using the interpolation theory and the approximation result
\begin{equation*}
     \vert \psi-G_h\psi\vert_{H^1(\Omega)}\lesssim h^r,
 \end{equation*}
we deduce
\begin{align*}
 	\Vert G_hg_1\Vert_{L^1(\Omega)}&=(G_hg_1,\varphi)\\\notag
  &=a_h(G_hg_1,G_h\psi-I_h \psi)+a_{h}(G_hg_1,I_h\psi)\\
  &=a_h(G_hg_1,G_h\psi-\psi)+a_h(G_hg_1,\psi-I_h \psi)+a(g_1,I_h\psi)\\
 	&\lesssim h^r\vert G_hg_1\vert_{H^1(\Omega)}+\vert g_1\vert_{W^{1,p}(\Omega)}\vert \psi\vert_{W^{1,p^*}(\Omega)},
 \end{align*}
 where $\frac{1}{p} +\frac{1}{p^*}=1$.
Following the same argument as in the proof of  Lemma \ref{lem eta-eta*__}, we   arrive at 
 \begin{equation}\label{eta-eta*_NC}
     \eta^* \lesssim \eta+ h^{\frac{1}{3}}  \Vert e_h\Vert_{L^{\infty}(\Omega)}.
\end{equation}

% where we just need to notice that $\vert w- G_h w\vert_{1,\Omega}$ has the same accuracy as $\vert w- I_h w\vert_{1,\Omega}$
\par Now we may give the equivalence between $\Vert e_h\Vert_{L^{\infty}(\Omega)}$ and $\eta$.
\begin{theorem} There exists a positive number $h_0>0$ such that for any $0<h\leq h_0$, the following estimates hold
    \begin{equation*}
\eta\lesssim \Vert e_h\Vert_{L^{\infty}(\Omega)}\lesssim  \vert \ln{h}\vert\eta.
    \end{equation*}
\end{theorem}
\begin{proof}
 The reliability of $\eta$ is straightforward in view of \eqref{relation_NC}, \eqref{T-Th_NC} and \eqref{eta-eta*_NC}. For the efficiency, with the continuity of $u$, we have 
$$ h_{E}\Vert [\nabla u_h]\cdot \bm{\tau}_E\Vert_{L^{\infty}(E)}\lesssim  \Vert [u_h]\Vert_{L^{\infty}(E)}= \Vert [e_h]\Vert_{L^{\infty}(E)}\leq 2\Vert e_h\Vert_{L^{\infty}(\Omega)} \ \ \ \ \forall \  E\in \mathcal{E}_h,$$
where $\bm{\tau}_E$ is the tangential unit vector of $E$ given by a counter clockwise $90^{\circ}$ rotation of $\bm{n}_E$.
With a similar argument as in the proof of Theorem \ref{main result}, we may get the efficiency of $\eta$. \qed
\end{proof}

 \section{Numerical experiments}
In this section we use the $P_1$-conforming finite element to approximate the Laplacian eigenvalue problems over the L-shape domain and the slit domain. For each domain, we perform the following adaptive procedure: 
$$ \textbf{Solve}\rightarrow \textbf{Estimate} \rightarrow \textbf{Mark} \rightarrow \textbf{Refine}.$$
Specifically, start with an initial uniform mesh $\mathcal{T}_0$ and obtain the approximate eigenfunction $u_0$. Then, given the approximate solution $u_j,\ j=0,1,2,\cdots$ ($j$ denotes the number of mesh level), we get the next mesh by refining the element $K$ satisfying:
\begin{equation}
    \eta_K\geq \theta \eta,
\end{equation}
where $0<\theta<1$ is a fixed positive constant (we  take $\theta=0.7$). To get the admissible mesh, we adopt the refining method in \cite{refinement}. For simplicity, we denote the degrees of freedom on the $j$th mesh level by $N_j$.

\begin{example}
We use the $P_{1}$-conforming finite element to approximate the Laplace eigenvalue problem in the L-shaped domain $(0,\pi)^2 \backslash ([\frac{\pi}{2},\pi)\times (0,\frac{\pi}{2}])$. We test the AFEM for computing the first eigenpair. Set the initial mesh size $h_0= \frac{\pi}{2^3}$ and the corresponding degrees of freedom $N_0=33$. The adaptive procedure stops when the degrees of freedom exceed $4\times 10^6$ and the final degrees of freedom equal to $4299472$.
\end{example}

\begin{table}[H]
    \centering    
    \begin{tabular}{cccc}\hline
       \textbf{mesh level}  & \textbf{degrees of freedom}&\textbf{a posteriori error}&\textbf{product} \\
        $j$& $N_j$& $\eta$&$ N_j\times \eta$\\
        \hline

0&33&	8.9705E-01&	29.6027\\
5&122&	2.1379E-01	&26.0821\\
10&408&	6.1547E-02	&25.1110\\
15&1312&	2.0559E-02&	26.9734\\
20&3394	&8.1732E-03&	27.7397\\
25&8987	&3.2335E-03&	29.0591\\
30&24895&	1.0904E-03&	27.1443\\ \hline
    \end{tabular}
    \caption{Degrees of freedom, the a posteriori error and their product in L-shape domain $(0,\pi)^2 \backslash ([\frac{\pi}{2},\pi)\times (0,\frac{\pi}{2}])$ for the first eigenpair on different mesh levels.}
    \label{tab:L1}
\end{table}

\par %From table \ref{tab:L1}, we may see that the product of $N$ and $\eta$ is stable near a certain constant.
Figure \ref{fig:L1mesh} shows the adaptive mesh with $7822$ elements, from which we can see that our adaptive algorithm captures the singularities accurately. Figure \ref{fig:L1rate} shows that the convergence rate of $\eta$ is $N^{-1}$, which is quasi-optimal. Moreover, from Table \ref{tab:L1} we may see that the product of $N_j$ and $\eta$ is gradually stable around a certain constant, which illustrates that $\eta$ is $O(N^{-1})$.

\begin{figure}[H]
    \centering  
    \includegraphics[width=7.2cm, height=5.4cm]{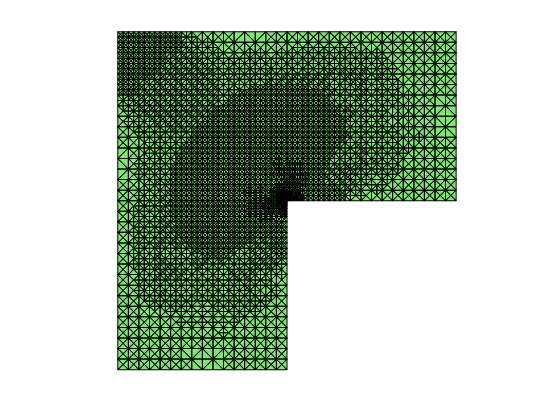}
    \caption{The adaptive mesh with $7822$ elements in L-shape domain $(0,\pi)^2 \backslash ([\frac{\pi}{2},\pi)\times (0,\frac{\pi}{2}])$ for the first eigenpair.}
    \label{fig:L1mesh} 
 \end{figure}

\begin{figure}[H]
    \centering  
    \includegraphics[width=7.2cm, height=5.4cm]{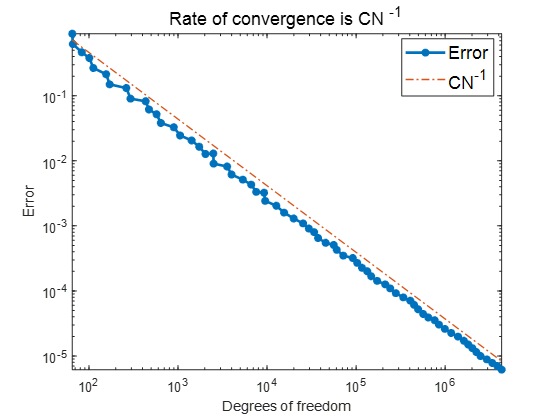}
    \caption{The  blue dots represent the pointwise a posteriori error estimator $\eta$ for the first eigenpair in L-shape domain $(0,\pi)^2 \backslash ([\frac{\pi}{2},\pi)\times (0,\frac{\pi}{2}])$. The dotted line represents the fitted convergence rate.}
    \label{fig:L1rate}
\end{figure}

\begin{example}
We consider the Laplace eigenvalue problem in the slit domain $\Omega=\{(x,y): \vert x\vert +\vert y\vert <1\}$ with the crack $[0,1]\times \{0\}$. We compute the first eigenpair with the AFEM. Set the initial mesh size $h_0= \frac{1}{2^3}$ and the corresponding degrees of freedom $N_0=105$. The adaptive procedure stops when the degrees of freedom exceed $4\times 10^6$ and the final degrees of freedom equal to $4530260$. 
\end{example}

\begin{table}[H]
    \centering
    \begin{tabular}{cccc}\hline
       \textbf{mesh level}  & \textbf{degrees of freedom}&\textbf{a posteriori error}&\textbf{product} \\
        $j$& $N_j$& $\eta$&$ N_j\times \eta$\\
        \hline
20&1936&	2.9157E-02&	56.4472\\
25&4884	&1.4588E-02	&71.2492\\
30&10152&	5.9188E-03	&60.0880\\
35&26280&	2.5793E-03	&67.7840\\
40&49637&	1.2864E-03	&63.8526\\
45&106367&	6.4315E-04	&68.4096\\
50&196954&	3.2170E-04	&63.3602\\ \hline
    \end{tabular}
    \caption{Degrees of freedom, the a posteriori error and their product in  slit domain $\Omega=\{(x,y): \vert x\vert +\vert y\vert <1\}$ with
the crack $[0,1]\times \{0\}$ for the first eigenpair on different mesh levels.}
    \label{tab:C1}
\end{table}

\begin{figure}[H]
    \centering  
    \includegraphics[width=10cm, height=7.5cm]{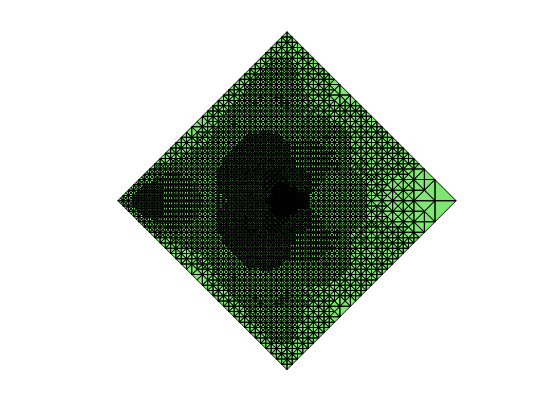}
    \caption{The adaptive mesh with $9971$ elements in slit domain $\Omega=\{(x,y): \vert x\vert +\vert y\vert <1\}$ with the crack $[0,1]\times \{0\}$ for the first eigenpair.}
    \label{fig:C1mesh}
\end{figure}

\begin{figure}[H]
    \centering  \includegraphics[width=8cm, height=6cm]{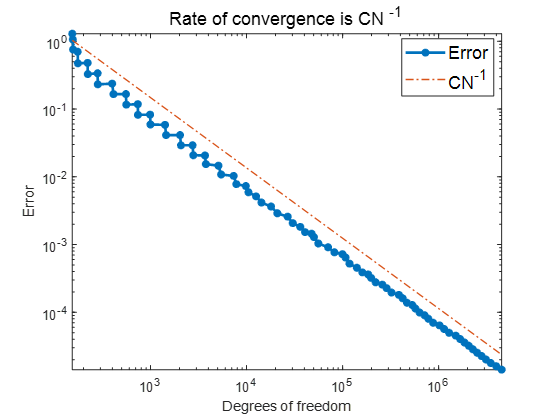}
    \caption{The blue dots represent the pointwise a posteriori error estimator $\eta$ for the first eigenpair in  slit domain $\Omega=\{(x,y): \vert x\vert +\vert y\vert <1\}$ with the crack $[0,1]\times \{0\}$. The dotted line represents the fitted convergence rate.}
    \label{fig:C1rate}
\end{figure}
Figure \ref{fig:C1mesh} shows the adaptive mesh with $9971$ elements. We can see that our adaptive algorithm captures the singularities accurately.  Figure \ref{fig:C1rate} shows that the convergence rate of $\eta$ is $N^{-1}$, which is quasi-optimal. Similarly, from Table \ref{tab:C1} we may see that the product of $N_j$ and $\eta$ is gradually stable around a certain constant, which illustrates that $\eta$ is $O(N^{-1})$.
\section{Conclusion}
In this paper, we propose a pointwise a posteriori error estimator for elliptic eigenvalue problems with AFEMs. It is proved that the pointwise a posteriori error estimator is reliable and efficient. Numerical results verify our theoretical findings. 

\appendix
\section{Appendix}
In the appendix, we give the proof of \eqref{smooth2} and Lemma \ref{Baohe}.
Before giving the proof of \eqref{smooth2}, we first review a pointwise estimate  for Green's function in \cite{5}.
\begin{lemma}
   Given a polygonal domain  $\Omega$ with vertices $z_j$ ($j=1,2,...,M$), let $G=G(x,y)$ be Green's function satisfying for any $y\in \Omega$,
   $$ -\Delta_x G(x,y)=\delta(x-y) \ \ \ \  x \in \Omega,\ \ G(x,y)=0 \ \ \ \ x\in \partial \Omega,$$
   where $\delta$ is the standard dirac delta function.
   Then we have pointwise estimates for its derivatives
   \begin{equation}\label{estimate_G}
      \vert \nabla_x^i G(x,y)\vert \lesssim\vert x-y\vert^{-i}+\sum_{j=1}^M \vert x-z_j\vert^{\gamma_j-i}\ \ \ \ i=1,2,
   \end{equation}
   where $\gamma_j=\frac{\pi}{\theta_j}$ with $\theta_j$ being the interior angle at $z_j$ and  $\nabla^i_x$ denotes the $i$th order derivative with respect to $x$.
\end{lemma}

\par \noindent{\bf Proof of \eqref{smooth2}}:
    Let $\Omega_0=\Omega \backslash B$, where $B$ is defined in \eqref{prodelta1}. We divide $\vert g_0\vert_{W^{2,1}(\Omega)}$ into two parts $\vert g_0\vert_{W^{2,1}(B)}$ and $\vert g_0\vert_{2,1,\Omega_0}$. First for a fixed $p_0\in (1,\frac{4}{3})$, with  the H{\"o}lder inequality and \eqref{priori}, we have 
    \begin{align*}
        \vert g_0\vert_{W^{2,1}(B)} \leq \vert B\vert ^{\frac{1}{q_0}}\vert g_0\vert_{W^{2,p_0}(\Omega)}\lesssim \vert B\vert ^{\frac{1}{q_0}}\Vert \delta_0\Vert_{L^{p_0}(B)}\lesssim 1,
    \end{align*}
    where $\frac{1}{p_0}+\frac{1}{q_0}=1$. Then we estimate $\vert g_0\vert_{W^{2,1}(\Omega_0)}$.
\par Using the representation of Green’s function and \eqref{estimate_G}, we obtain 
\begin{align*}
\vert g_0\vert_{W^{2,1}(\Omega_0)}&\leq \int_{\Omega_0}\int_{B}\vert \nabla_x^2 G(x,y)\vert \delta_0(y)\ dydx\\
&\lesssim \int_{\Omega_0}\int_{B} \left(\vert x-y\vert^{-2}+\sum_{j=1}^M \vert x-z_j\vert^{\gamma_j-2} \right)\delta_0(y)\ dydx.
\end{align*}
Let 
$$\text{I}_1:= \int_{\Omega_0}\int_{B} \vert x-y\vert^{-2}\delta_0(y)\ dydx$$
and 
$$ \text{I}_2:=\int_{\Omega_0}\int_{B} \sum_{j=1}^M \vert x-z_j\vert^{\gamma_j-2} \delta_0(y)\ dydx.$$
By a direct computation, we have 
$$\text{I}_2\leq \Vert \delta_0\Vert_{L^1(B)}\int_{\Omega_0} \sum_{j=1}^M \vert x-z_j\vert^{\gamma_j-2} dx\lesssim 1.$$
To get the estimate for $\text{I}_1$ we use a dyadic decomposition of $\Omega_0$. Denote $d_j=2^j \rho$ for $j=0,1,...$ and define $A_j=\{x\in \Omega_0: d_{j-1}< \vert x-x_0\vert \leq d_j\}$ for $j\geq 1$. Then we choose $J$ such that $\text{diam}(\Omega)\leq d_J< 2\text{diam}(\Omega)$. Note that $\vert A_j\vert \lesssim  d_j^2$ and $J\lesssim \vert \ln{\rho}\vert$. We may deduce 
\begin{align*}
    \text{I}_1\leq \sum_{j=1}^J \vert A_j\vert d_{j-1}^{-2}\Vert \delta_0\Vert_{L^1(B)}\lesssim \vert \ln{\rho}\vert.
\end{align*}
 Adding the estimates for $\vert g_0\vert_{W^{2,1}(\Omega_0)}$ and $\vert g_0\vert_{W^{2,1}(B)}$ leads to \eqref{smooth2}. \hfill \qed

\par Next, we give the proof of Lemma \ref{Baohe}.
\par \noindent{\bf Proof of Lemma \ref{Baohe}}:
It is well known that $u\in C^{\infty}(U)$ for any $U\subset\subset \Omega$, , where $U\subset\subset \Omega$ represents that the closure of $U$ is contained in $\Omega$. Let ${y_0}\in \Omega$ satisfy $\vert u({y_0})\vert =\Vert u\Vert_{L^{\infty}(\Omega)}$. In view of the definition of $u$, we have 
\begin{equation}\label{y0}
    \vert \Delta u({y_0})\vert =\lambda \Vert u\Vert_{L^{\infty}(\Omega)}\gtrsim 1. 
\end{equation}
Since $u$ is smooth enough in the interior of $\Omega$, we may select an open set $U\subset \subset \Omega$ containing $y_0$ satisfying $\Vert u\Vert_{W^{3,\infty}(U)}\lesssim 1$. It is clear that the eigenvalues of the Hessian matrix of $u$, denoted by $s_1(x)$, $s_2(x)$ (for notational convenience, we denote $s_{1}=s_1(x)$ and $s_{2}=s_2(x)$),  satisfy 
 $$s_1 +s_2=\Delta u.$$ 
Then at least one of them (without the loss of generality, we assume it is $s_1$) satisfies 
\begin{equation}\label{eig}
     \vert s_1(y_0)\vert \geq \frac{1}{2} \vert \Delta u(y_0)\vert \gtrsim 1. 
\end{equation}
Let ${v_1}$ denote the eigenvector associated with $s_1$ and $K^*\subset U$ be an element containing $y_{0}$ in $\mathcal{T}_{h}$ (we can find $K^*$ when $h_{K^*}$ is small enough). Then there exist three points ${y_i}$ ($1\leq i \leq 3$) on $\partial {K^*}$  satisfying the following properties:
\begin{itemize}
\item[(i)] $y_1$ is some vertex of $K^*$;\\
\item[(ii)]
$y_1-y_2$ is parallel to $v_1(y_0)$ (specially, if $v_1(y_{0})$ is parallel to some edge of $K^*$ containing $y_1$, we take $y_2$ the other vertex of this edge).\\
\item[(iii)] Let $z_1$ and $z_2$ be vertices besides $y_1$ of the triangle $K^*$  (without the loss of generality, we assume $z_1$ is further away from $y_2$). Then $y_3=z_1$.
\end{itemize}
\par We define an local interpolation $\Tilde{I}_h$ (note that we assume any $K\in\mathcal{T}_h$ is a closed set) by 
\begin{align*}
   \Tilde{I}_hv=\begin{cases}
       \sum_{i=1}^3 v({y_i})\Tilde{\phi}_i, &x\in K^*\\
       0, &x\in \overline{\Omega}\backslash K^*
   \end{cases}
\end{align*}
for $v\in C^0(K^*)$,
where $\Tilde{\phi}_i\in P_1(K^*)$  satisfies $\Tilde{\phi}_i(y_j)=\delta_{ij}$, and $\delta_{ij}$ denotes the Kronecker notation. With the aid of \eqref{eig} and noting that $y_1-y_2$ is parallel to $v_1(y_0)$, we use the Taylor's expansion at $\frac{{y_1}+{y_2}}{2}$ and get 
\begin{align*}
    \vert u({y_1})+u({y_2})-2u(\frac{{y_1}+{y_2}}{2})\vert &\gtrsim h_{K^*}^2- \Vert u\Vert_{W^{3,\infty}(U)}h_{K^*}^3\\
    &\gtrsim (1-h_{K^*})h_{K^*}^2.
\end{align*}
Moreover, we obtain
\begin{equation*}
     \vert u({y_1})+u({y_2})-2u(\frac{{y_1}+{y_2}}{2})\vert \gtrsim h_{K^*}^2
\end{equation*}
when $h_{K^*}$ is small enough. Noting that $u({y_i})=\Tilde{I}_hu({y_i})$ ($i=1,2$), with the linearity of $\Tilde{I}_hu$ on $K^*$ , we have
\begin{align*}
    \vert u(\frac{{y_1}+{y_2}}{2})-\Tilde{I}_hu(\frac{{y_1}+{y_2}}{2})\vert&=\frac{1}{2}\vert u({y_1})+u({y_2})-2u(\frac{{y_1}+{y_2}}{2})\vert \\
    &\gtrsim h_{K^*}^2.
\end{align*}
The properties (i), (ii) and (iii) guarantee the boundedness of $\Tilde{I}_h$ in $L^{\infty}$-norm, namely,
\begin{equation*}
  ||\Tilde{I}_h v||_{L^{\infty}(\Omega)}=  ||\Tilde{I}_h v||_{L^{\infty}(K^*)}\leq ||v||_{L^{\infty}(K^*)}.
\end{equation*}
Therefore, we may deduce 
\begin{align*}
        h_{K^*}^2&\lesssim \Vert u-\Tilde{I}_hu\Vert_{L^{\infty}(K^*)}=\Vert e_h-\Tilde{I}_he_h\Vert_{L^{\infty}(K^*)}\\ 
        &\lesssim \Vert e_h\Vert_{L^{\infty}(K^*)}\leq \Vert e_h\Vert_{L^{\infty}(\Omega)}, 
\end{align*} 
which completes the proof.  \hfill \qed
% \end{itemize}

\begin{small}
\bibliographystyle{plain}
\bibliography{reference}
\end{small}
\end{document}